\documentclass[letterpaper, abstract=true]{scrartcl}
\pdfoutput=1

\usepackage[utf8]{inputenc}
\usepackage{fullpage, amsthm, amssymb, amsmath, mathtools, xcolor}

\usepackage[onehalfspacing]{setspace}

\usepackage[affil-sl]{authblk}

\usepackage[%
  backend=biber,%
  style=numeric,%
  natbib=true,%,
  url=true,%
  doi=true,%
  eprint=true,%
  backref=true,%
  maxbibnames=99,%
]{biblatex}

\DeclareFieldFormat{eprint:mr}{% based on eprint:jstor
  \space \textsc{mr}\addcolon\space
  \ifhyperref
    {\href{http://www.ams.org/mathscinet-getitem?mr=MR#1}{\nolinkurl{#1}}}
    {\nolinkurl{#1}}}

\DeclareSourcemap{
  \maps[datatype=bibtex]{
    \map{
      \step[fieldsource=mrnumber, fieldtarget=eprint, final]
      \step[fieldset=eprinttype, fieldvalue=mr]
    }
  }
}
  
\usepackage[unicode]{hyperref}
\usepackage[nameinlink,capitalize]{cleveref}

\hypersetup{
    colorlinks=true,
    linkcolor=blue,
    filecolor=blue,      
    urlcolor=blue,
    citecolor=blue
}

\newtheoremstyle{plain}
  {3mm}                         % Space above theorem and previous line.
  {3mm}                         % Space below theorem box and next line.
  {\slshape}                    % Use slanted font in body of theorem.
  {}                            % Indent amount from margin. (Here no
  {\bfseries}       % Theorem head font.
  {.}                           % Punctuation after theorem head.
  {.5em}                        % Space after theorem head.
  {}                            % Theorem head specification.

\newtheoremstyle{definition}
  {3mm}                         % Space above theorem and previous line.
  {3mm}                         % Space below theorem box and next line.
  {}                            % Use slanted font in body of theorem.
  {}                            % Indent amount from margin. (Here no
  {\bfseries}       % Theorem head font.
  {.}                           % Punctuation after theorem head.
  {.5em}                        % Space after theorem head.
  {}                            % Theorem head specification

\theoremstyle{plain}
\newtheorem{theorem}{Theorem}[section]
\newtheorem{lemma}[theorem]{Lemma}
\newtheorem{proposition}[theorem]{Proposition}
\newtheorem{corollary}[theorem]{Corollary}

\theoremstyle{definition}
\newtheorem{definition}[theorem]{Definition}
\newtheorem{remark}[theorem]{Remark}
\newtheorem{example}[theorem]{Example}

\newtheorem{problem}[theorem]{Problem}
\newtheorem{question}[theorem]{Question}

\newcommand{\define}[1]{{\textsl{\textbf{#1}}}}

\title{Algebraic characterizations of some relative notions of size}
\author{Cory Christopherson}

\author{John H.~Johnson Jr.\\
  \small%
  \href{mailto:johnson.5316@osu.edu}{johnson.5316@osu.edu}}

\affil{%
  \small%
  Department of Mathematics\\
  The Ohio State University\\
  Columbus, Ohio
}

\begin{document}
\maketitle

\begin{abstract}
	We obtain algebraic characterizations of relative notions of size in a discrete semigroup that generalize the usual combinatorial notions of syndetic, thick, and piecewise syndetic sets.
	``Filtered'' syndetic and piecewise syndetic sets were defined and applied earlier by  Shuungula, Zelenyuk, and Zelenyuk \cite{Shuungula:2009ty}.
	Other instances of these relative notions of size have appeared explicitly (and more often implicitly) in the literature related to the algebraic structure of the Stone--\v{C}ech compactification. 
	Building on this prior work, we observe a natural duality and demonstrate how these notions of size may be composed to characterize previous notions of size (like piecewise syndetic sets) and serve as a convenient description for new notions of size.
	
	\vspace{2em}

	\vspace{1em}

 \end{abstract}
\section{Introduction}
\label{section:introduction}
In this paper, we study certain ``relative notions of size'' that can be succinctly characterized by the algebraic structure of the Stone--\v{C}ech compactification of a discrete semigroup.
All notions of size we consider are ultimately motivated by van der Waerden's theorem on arithmetic progressions~\cite{vanderWaerden1927}.
This well known classical result in Ramsey theory has many equivalent formulations.
But, one early observation (see~\cite[Chapter 33, p.~319]{Soifer2009}) of Kakeya and Morimoto~\cite[Sections~1 and 2]{KAKEYA1930} shows that van der Waerden's theorem can be reformulated to assert that certain ``large'' subsets of positive integers contain arbitrarily long arithmetic progressions:

\begin{theorem}[Reformulation of van der Waerden's theorem]
\label{theorem:vdw}
Let \( A \subseteq \mathbb{N} \) have \define{bounded gaps}, that is, there exists a positive integer \( b \) such that for all positive integers \( x \) we have 
\[
    (\{1, 2, \ldots, b\} + x) \cap A \ne \emptyset.
\]
Then for all positive integers \( k \in \mathbb{N} \) there exist \( a, d \in \mathbb{N} \) with
\[
    \{ a, a+d, a+2d, \ldots, a+(k-1)d \} \subseteq A.
\]
\end{theorem}

This reformulation is important since it's the first \emph{documented} example (known to the authors) of a fundamental heuristic that underlies a significant proportion of current research in Ramsey theory.
Roughly stated, this heuristic asserts that underlying many Ramsey theoretic phenomena is at least one \textbf{notion of size} (for instance, bounded gaps) which contains enough ``\textbf{structure}'' to imply an interesting (\textbf{combinatorial}) \textbf{pattern} (for instance, arbitrarily long arithmetic progressions).
In \cref{theorem:vdw} one can use the structure of minimal left ideals in the Stone--\v{C}ech compactification, as first shown by Bergelson, Furstenberg, Hindman, and Katznelson \cite{Bergelson:1989aa}, to obtain an algebraic proof of van der Waerden's theorem. 
Often the most difficult part of this trio is to identify the right structures to leverage to deduce the combinatorial consequence.

\begin{remark}
\label{remark:ramsey-theory-heuristic}
Historically, however, the dissemination of the ``Ramsey theory heuristic'' was more directly influenced from developments surrounding the Erd\H{o}s and Turán Conjecture \cite{Erdos1936} leading up to Szemerédi's Theorem \cite{Szemeredi1975} and beyond.
Tao's essay \cite{Tao2007a} provides a nice brief nontechnical overview on how Szemerédi's Theorem, and especially its proofs, have influenced and motivated a lot of work in Ramsey theory and its applications.
Additionally, Bergelson's survey article \cite{MR2757532} provides several examples and theorems demonstrating the effectiveness of approaching Ramsey theory from this heuristic.    
\end{remark}

\subsection*{Three notions of size in semigroups}
\label{section:notion-of-size-semigroups}
In general, given a subset \( A \) of a (for us, infinite) semigroup \( S \), we can classify \( A \) as ``large'' in several different ways.
A recent paper of Hindman~\cite{Hindman2019a} surveys 52 notions of size which have interesting connections to either Ramsey theory, dynamics, or the algebraic structure of the Stone--\v{C}ech compactification of a discrete semigroup. 
However in this paper we'll consider far fewer notions~---~essentially \emph{only three}.%, and, simultaneously, consider many more, usually at least \( 2^{2^{|S|}} \).

For an arbitrary semigroup a set with bounded gaps is not guaranteed to make sense because most semigroups don't have a natural ordering. 
But, we can capture the essential properties of bounded gaps via the notion of ``syndetic''.
In the following definition, and in the rest of this paper, given a set \( X \) we let \( \mathcal{P}_f(X) \) denote the collection of all nonempty finite subsets of \( X \); and,
if \( (S, \cdot) \) is a semigroup, \( A \subseteq S \), and \( x \in S \) we define \( x^{-1}A = \{ y \in S : x \cdot y \in A \} \).
Beside defining syndetic, we also introduce two more notions closely related to syndetic sets:
\begin{definition}
\label{definition:syndetic-and-thick}
Let \( (S, \cdot) \) be a semigroup and let \( A \subseteq S \).
\begin{itemize}
    \item[(a)]
    We call \( A \) \define{syndetic} if and only if there exists \( H \in \mathcal{P}_f(S) \) such that \( \bigcup_{h \in H} h^{-1}A = S \).
    We let \( \mathsf{Syn} \) denote the collection of all syndetic subsets of \( S \).

    \item[(b)]
    We call \(A \) \define{thick} if and only if for all \( H \in \mathcal{P}_f(S) \) we have \( \bigcap_{h \in H} h^{-1}A \ne \emptyset \).
    We let \( \mathsf{Thick} \) denote the collection of all thick subsets of \( S \).
    
    \item[(c)]
    We call \( A \) \define{piecewise syndetic} if and only if there exist \( B \in \mathsf{Syn} \) and \( C \in \mathsf{Thick} \) such that \( A = B \cap C \).
    We let \( \mathsf{PS} \) denote the collection of all piecewise syndetic subsets of \( S \).
\end{itemize}
\end{definition}

We leave it as an exercise to verify, in \( (\mathbb{N}, +) \) syndetic sets are precisely those sets with bounded gaps.
It's also helpful to keep in mind that, again in \( (\mathbb{N}, +) \), a set is thick if and only if it contains arbitrarily long blocks of consecutive positive integers, and a set is piecewise syndetic if and only if there exists a fixed bound such that the set contains arbitrarily long subsets of positive integers whose gaps are no bigger than the fixed bound.
Furstenberg in \cite[Definition~1.11]{Furstenberg1981a} defined piecewise syndetic sets, in \( \mathbb{N} \) or \( \mathbb{Z} \), as the intersection of a syndetic and thick set (and also noted the equivalence we just stated above).

In \cref{section:preliminaries}, we'll give a brief review of the algebraic structure of the Stone--\v{C}ech compactification \( (\beta S, \cdot) \) of a discrete semigroup \( (S, \cdot) \), but, for now, we'll simply note that all three notions have succinct characterizations in terms of this algebraic structure \cite[Theorems~4.48 and 4.40]{Hindman:2012tq}:
\begin{theorem}
    \label{theorem:algebraic-characterizations}
    Let \( (S, \cdot) \) be a semigroup.
    \begin{itemize}
        \item[(a)]
        \( \mathsf{Syn} = \{ A \subseteq S : \text{ for every minimal left ideal } L \text{ of } \beta S \text{ we have } L \cap c\ell_{\beta S}(A) \ne \emptyset \} \). 
        
        \item[(b)]
        \( \mathsf{Thick} = \{ A \subseteq S : \text{ there exists a minimal left ideal } L \text{ of } \beta S \text{ such that } L \subseteq c\ell_{\beta S}(A) \} \). 
        
        \item[(c)]
        \( \mathsf{PS} = \{ A \subseteq S : \text{ there exists a minimal left ideal } L \text{ of } \beta S \text{ such that } L \cap c\ell_{\beta S}(A) \ne \emptyset \} \).         
    \end{itemize}
\end{theorem}

\subsection*{Organization of article}
\label{section:organization}
The goal of this paper is to define and study ``relative'' notions of syndetic, thick, and piecewise syndetic sets by investigating how these notions ``compose'' with each other and prove algebraic characterizations of these relative notions that generalize \cref{theorem:algebraic-characterizations}.

\paragraph{Relative notions in the literature}``Filtered'' notions of syndetic and piecewise syndetic sets were previously defined and considered by Shuungula, Zelenyuk, and Zelenyuk \cite{Shuungula:2009ty}.
Their paper, and a related older paper of Davenport \cite{Davenport:1990wq}, both form the starting part for our own investigations.
(In a sense, this paper can be partially viewed as a synthesis of their work.)
``Filtered'' notions of thick sets have also appeared implicitly in much of the literature related to the algebraic structure of the Stone--\v{C}ech compactification.
And, in special cases, appeared more or less explicitly in the context of `finite embeddability' by Blass and Di Nasso and Baglini \cite{Blass2016, LuperiBaglini2014, LuperiBaglini2016} and in a note of Protasov and Slobodianiuk \cite{Protasov:2015uz}.
We also note that Zucker \cite{Zucker2017} considers some related ideas in the context of a different generalization of syndetic, thick, and piecewise syndetic sets.

In \cref{section:preliminaries}, we \hyperref[prop:mesh-operator]{state standard and known results} in \cref{prop:mesh-operator} and \hyperref[definition:mesh-operator]{fix some notation} and \hyperref[definition:stacks-filters-grills-ultrafilters]{terminology}. 
(As we note in \cref{remark:origin-of-nonstandard-terminology} and \cref{remark:mesh-operator} some of the notation and terminology we introduce, while previously appearing in the literature, is not standardized.
But, our choices are suitable and flexible for our purpose.)
We also end this section with a \hyperref[section:algebraic-review]{brief review of the algebraic structure of the Stone--\v{C}ech compactification}. 

In \cref{section:relative-syndetic-and-thick}, we give the \hyperref[definition:relative-syndetic-and-thick]{definitions for relative syndetic and thick sets}, \hyperref[proposition:duality-principle]{observe a duality between them}, \hyperref[proposition:characterization-of-PS-as-broken-syndetic]{illustrate how these notions can be ``composed''}, and prove, in \cref{lemma:algebraic-characterizations}, \hyperref[lemma:algebraic-characterizations]{a result that algebraically characterizes a large number of these notions}.
As \hyperref[corollary:combinatorial-characterizations]{one (almost immediate) consequence of this lemma}, we obtain \hyperref[corollary:combinatorial-characterizations]{combinatorial characterizations of closed subsemigroups, closed left ideals, and closed right ideals of \( \beta S \)}.
These latter combinatorial results were previously obtained by Davenport \cite{Davenport:1990wq} and, independently, by Papazyan \cite{Papazyan1990}.

In \cref{section:relative-piecewise-syndetic}, we \hyperref[definition:piecewise-FG-syndetic]{define relative piecewise syndetic sets} and immediately \hyperref[theorem:brown-lemma-relative-piecewise-syndetic]{observe this notion is `partition regular'} in \cref{theorem:brown-lemma-relative-piecewise-syndetic}.
Our definition of relative piecewise syndetic is different from the one defined earlier by Shuungula, Zelenyuk, and Zelenyuk \cite{Shuungula:2009ty}.
We show that our \hyperref[corollary:relative-PS-union-are-thick]{definition satisfies their notion of relative piecewise syndetic}.
However, since we've been unable to verify the converse, we end this section with an \hyperref[question:converse-PS]{open question asking if these two definitions, in fact, coincide}. 

We also note that in several of our proofs below given three (or more) statements \( P \), \( Q \), and \( R \), following the usual mathematical practice, when we write ``\( P \iff Q \iff R \)'' we mean ``\( (P \iff Q) \) and \( (Q \iff R) \)''.%  

\section{Preliminaries: Notions of size and the Stone--\v{C}ech compactification}
\label{section:preliminaries}
Instead of considering a single ``large'' set, it's usually convenient to consider a collection of all large subsets, as we've done in \cref{definition:syndetic-and-thick} via \( \mathsf{Syn} \), \( \mathsf{Thick} \), and \( \mathsf{PS} \).
Moreover, it's reasonable to assume that such a collection satisfies some minimal requirements.
(For example, any such collection should be nonempty and not contain the empty set.)
To this end, we introduce some terminology defined using four conditions, that, in a sense, axiomatizes ``notions of size'' in set-theoretic terms: 
\begin{definition}
	\label{definition:stacks-filters-grills-ultrafilters}
    Let \( X \) be a nonempty set and let \( \mathcal{F} \subseteq \mathcal{P}(X) \).
    \begin{itemize}
        \item[(a)]  % Definition of a stack on a fixed set.
         We call \( \mathcal{F} \) a \define{stack on \( X \)} if and only if \( \mathcal{F} \) satisfies two conditions:
         \begin{itemize}
             \item[(1)] %\textit{(Nontrivial)}
             \( \emptyset \ne \mathcal{F} \) and \( \emptyset \not\in \mathcal{F} \) and
             \item[(2)] %\textit{(Closed under supersets)}
             \( A \in \mathcal{F} \) and \( A \subseteq B \subseteq X \) implies \( B \in \mathcal{F} \).
         \end{itemize}

    \item[(b)] % Definition of a filter on a fixed set.
      We call \( \mathcal{F} \) a \define{filter on \( X \)} if and only if \( \mathcal{F} \) is a stack and \( \mathcal{F} \) satisfies
      \begin{itemize}
        \item[(3)] %\textit{(closed under finite intersections)}
         \( A \in \mathcal{F} \mbox{ and } B \in \mathcal{F} \text{ implies } A \cap B \in \mathcal{F} \).
      \end{itemize}
            
    \item[(c)] % Definition of a grill on a fixed set.
      We call \( \mathcal{F} \) a \define{grill on \( X \)} if and only if \( \mathcal{F} \) is a stack and \( \mathcal{F} \) satisfies
      \begin{itemize}
        \item[(4)] %\textit{(partition regular)}
         \( A \cup B \in \mathcal{F} \mbox{ implies } A \in \mathcal{F} \mbox{ or } B \in \mathcal{F} \).
      \end{itemize}
      
    \item[(d)] % Definition of an ultrafilter on a fixed set.
      We call \( \mathcal{F} \) an \define{ultrafilter on \( X \)} if and only if \( \mathcal{F} \) is both a filter and a grill.

    \end{itemize}
\end{definition}

\begin{remark}[\textit{Origin of some of the nonstandard terminology}]
  \label{remark:origin-of-nonstandard-terminology}
  We take ``stack'' from ``\textit{Stapel}'' of \cite[p.~321]{grimeisen_gefilterte_1960} since it is relatively short and somewhat descriptive.
  Stacks are also called, 
  especially in the dynamics literature,
  ``(Furstenberg) families''.
  (For example, see \cite[Introduction and Chapter~2]{Akin1997}.)
  
  The terms ``filter'' and ``ultrafilter'' are both well known and completely standard, 
  but the term ``grill'' seems like a reasonable standard but is not, perhaps, well known.
  This latter term we take from ``\textit{grille}'' in \cite{choquet_sur_1947}.
%  Choquet introduced this notion to produce a ``parameterization'' of the set of all filters on a nonempty set $X$ and noted some connections between grills and topology.
\end{remark}

From \cref{definition:syndetic-and-thick} it's easy to verify \( \mathsf{Syn} \), \( \mathsf{Thick} \), and \( \mathsf{PS} \) are all stacks on \( S \).
Simple examples in \( (\mathbb{N}, +) \) show that neither the collections \( \mathsf{Syn} \) or \( \mathsf{Thick} \) are guaranteed to be either a filter or grill, however \( \mathsf{PS} \) \emph{is a grill}.

Showing that \( \mathsf{PS} \) is a grill is not trivial (but not too hard either).
In fact, the assertion that for the semigroup \( (\mathbb{N}, +) \) the collection \( \mathsf{PS} \) is a grill is (commonly referred to as) Brown's lemma \cite[Lemma~1]{Brown1971}.
We refer you to \cite[Section~2]{Hindman2019a} for more details on the historical appearance of this notion.
In contrast to the algebraic \cite[a direct consequence of Theorem~4.40]{Hindman:2012tq}, the combinatorial  \cite[Theorem~2.5]{Bergelson1998}, and the dynamical  \cite[Theorem~1.24]{Furstenberg1981a} proofs that \( \mathsf{PS} \) is a grill, we'll provide another combinatorial proof of this fact in \cref{corollary:brown-lemma}.

The reason we introduce this additional terminology is that our point-of-view will be to think of stacks, filters, grills, and ultrafilters as each describing a different aspect of a notion of size.
Moreover, all four notions are connected by a certain operator on \( \mathcal{P}(\mathcal{P}(X)) \):
\begin{definition}%[\textit{Mesh operator on collection of subsets}]
  \label{definition:mesh-operator}
  Let \( X \) be a nonempty set and let \( \mathcal{F} \subseteq \mathcal{P}(X) \).
  The \define{mesh of \( \mathcal{F} \)} is \( \mathcal{F}^* = \{ A \subseteq X : X \setminus A \not\in \mathcal{F} \} \).
\end{definition}

\begin{remark}[\textit{No standard term for mesh operator}]
  \label{remark:mesh-operator}
  There is no standard nor well known terminology for what we call the ``mesh operator''.
  Perhaps the closest attempt, from which we derive our terminology, is ``\textit{Verzahnung}'' in \cite[Kapitel~II]{Schmidt:1953gm}.
  (Actually Schmidt \cite[Kapitel~II]{Schmidt:1953gm} defines his mesh operator as \( \mathcal{F}^* = \{ A \subseteq X : (\forall B \in \mathcal{F})\; A \cap B \ne \emptyset \}\). 
  When \( \mathcal{F} \) is a stack, both these definitions coincide (see \cref{prop:mesh-operator}(d).)  
  In the dynamics literature, \( \mathcal{F}^* \) is called a ``dual family'', and if \( \mathcal{F} \) is a filter, \( \mathcal{F}^* \) is called a \textit{filterdual} (see \cite[Chapter~2]{Akin1997}).
\end{remark}

As an example of using the mesh operator, first observe that there is a duality between syndetic and thick: \( A \subseteq S \) is syndetic if and only if \( S \setminus A \) is not thick.
This duality, using the mesh operator, can also be written as \( \mathsf{Syn} = \mathsf{Thick}^* \).
While this may seem like a triviality, we'll soon see this duality is a fundamental fact that helps show \( \mathsf{PS} \) is a grill (see \cref{corollary:brown-lemma} and its use of \cref{prop:mesh-operator}(h)).

The following proposition states some of the fundamental properties of the mesh operator.
A significant subset of these statements was noted at least as early as Choquet \cite{choquet_sur_1947}.
Schmidt \cite[Kapitel~II]{Schmidt:1953gm} also proves a significant subset of these statements.
A more contemporary reference for all of these statements is Akins \cite[Propositions~2.1, 2.2, and 2.3]{Akin1997}.
The reader can easily verify the following proposition, but, for completeness and convenience, we'll include the proof. %since our notation differs from either of the mentioned references.

\begin{proposition}%[\textit{Fundamental properties of mesh operator}]
  \label{prop:mesh-operator}
  Let \( \mathcal{F} \), \( \mathcal{F}_1 \), and \( \mathcal{F}_2 \) be stacks on \( X \).
  \begin{itemize}
    \item[(a)] \( \mathcal{F}^* \) is a stack on \( X \).   
    \item[(b)] \( \mathcal{F} = (\mathcal{F}^*)^* \).
    \item[(c)] \( \mathcal{F}_1 \subseteq \mathcal{F}_2 \) if and only if \( \mathcal{F}_2^* \subseteq \mathcal{F}_1^* \).
    \item[(d)] \( \mathcal{F} = \{ A \subseteq X : (\forall B \in \mathcal{F}^*) \; A \cap B \ne \emptyset \} \). %(Dually: \( \mathcal{F}^* = \{ A \subseteq X : (\forall B \in \mathcal{F}) \; A \cap B \ne \emptyset  \} \).)
    \item[(e)] \( \mathcal{F} \) is a filter if and only if \( \mathcal{F}^* \) is a grill.
    \item[(f)] If \( \mathcal{F} \) a filter, then \( \mathcal{F} \) is an ultrafilter if and only if \( \mathcal{F} = \mathcal{F}^* \).
    \item[(g)] If \( \mathcal{F} \) a filter and \( p \) is an ultrafilter, then \( \mathcal{F} \subseteq p \) if and only if \( p \subseteq \mathcal{F}^* \).
    \item[(h)] The collection \( \{ B \cap C : B \in \mathcal{F} \mbox{ and } C \in \mathcal{F}^* \} \) is a grill on \( X \) that contains \( \mathcal{F} \) and  \( \mathcal{F}^* \).
  \end{itemize}
\end{proposition}
\begin{proof}
    \textbf{(a)}
    Since \( X \setminus X = \emptyset \not\in \mathcal{F} \) and \( X \setminus \emptyset = X \in \mathcal{F} \) we have \( X \in \mathcal{F}^* \) and \( \emptyset \not\in \mathcal{F}^* \).
    Now let \( A \in \mathcal{F}^*  \) and \( A \subseteq B \subseteq X \).
    Since \( X \setminus B \subseteq X \setminus A \) and \( X \setminus A \not\in \mathcal{F} \) we have \( X \setminus B \not\in \mathcal{F} \) (since \( \mathcal{F} \) is a stack).
    Hence \( B \in \mathcal{F}^* \).
    This shows \( \mathcal{F}^* \) is a stack on \( X \).
    
    \textbf{(b)}
    This follows from definition: \( A \in (\mathcal{F}^*)^* \iff X \setminus A \not\in \mathcal{F}^* \iff A \in \mathcal{F} \).
    
    \textbf{(c)}
    By statement (b), it suffices to show \( \mathcal{F}_1 \subseteq \mathcal{F}_2 \) implies \( \mathcal{F}_2^* \subseteq \mathcal{F}_1^* \).
    Suppose \( \mathcal{F}_1 \subseteq \mathcal{F}_2 \) then \( A \in \mathcal{F}_2^* \), that is, \( X \setminus A \not\in \mathcal{F}_2 \) and its follows that \( X \setminus A \not\in \mathcal{F}_1 \).
    Hence \( A \in \mathcal{F}_1^* \).
    
    \textbf{(d)}
    We have \( A \in \mathcal{F} \iff X \setminus A \not\in \mathcal{F}^* \iff \neg ( \exists B \in \mathcal{F}^* )\; B \subseteq X \setminus A \iff (\forall B \in \mathcal{F}^*) \; B \cap A \ne \emptyset \), where the middle equivalence follows from statement (a).
    
    \textbf{(e)}
    By assumption and statement (a) we automatically know that \( \mathcal{F} \) and \( \mathcal{F}^* \) are both stacks on \( X \).
    
    First, suppose that \( \mathcal{F} \) is a filter and let \( A \cup B \in \mathcal{F}^* \), that is, \( (X \setminus A) \cap (X \setminus B) = X \setminus (A \cup B) \not\in \mathcal{F} \).
    It follows, since \( \mathcal{F} \) is a filter, that either \( X \setminus A \not\in \mathcal{F} \) or \( X \setminus B \not\in \mathcal{F} \), that is,     either \( A \in \mathcal{F}^* \) or \( B \in \mathcal{F}^* \).
    Therefore \( \mathcal{F}^* \) is a grill on \( X \).
    
    Now, suppose \( \mathcal{F}^* \) is a grill and let \( A \in \mathcal{F} \) and \( B \in \mathcal{F} \), that is, \( X \setminus A \not\in \mathcal{F}^* \) and \( X \setminus B \not\in \mathcal{F}^* \).
    It follows, since \( \mathcal{F}^* \) is a grill, that \( X \setminus (A \cap B) = (X \setminus A) \cup (X\setminus B) \not\in \mathcal{F}^* \), that is, \( A \cap B \in \mathcal{F} \).
    Therefore \( \mathcal{F} \) is a filter on \( X \).
    
    \textbf{(f)}
    Observe that when \( \mathcal{F} \) is a filter we have \( A \in \mathcal{F} \implies  X \setminus A \not\in \mathcal{F} \), and so \( \mathcal{F} \subseteq \mathcal{F}^* \).
    
    If \( \mathcal{F} \) is an ultrafilter, then \( \mathcal{F} \) is a grill and by statement (e) we have \( \mathcal{F}^* \) is a filter.
    Hence by our observation we have \( \mathcal{F}^* \subseteq \mathcal{F} \).
    If \( \mathcal{F}^* = \mathcal{F} \), then, since \( \mathcal{F} \) is a filter, by statement (e) we have \( \mathcal{F}^* \) is a grill and so \( \mathcal{F} \) is an ultrafilter.
    
     (Note that we need the assumption that \( \mathcal{F} \) is a filter.
     For if \( X = \{1, 2, 3\} \) and \( \mathcal{F} = \{A \subseteq X : |A| \ge 2\} \), then \( \mathcal{F}= \mathcal{F}^* \) but \( \mathcal{F} \) is not an ultrafilter, since it is not a grill.)
     
     \textbf{(g)}
     This follows directly from statements~(c) and (f).
     
     \textbf{(h)}
     Put \( \mathcal{G} = \{ B \cap C : B \in \mathcal{F} \text{ and }  C \in \mathcal{F}^* \} \) and observe that \(\mathcal{F} \subseteq \mathcal{G} \) implies \( \emptyset \ne \mathcal{G} \) and statement (d) implies \( \emptyset \not\in \mathcal{G} \).
     If \( A \in \mathcal{G} \), so \( A = B \cap C \) for some \( B \in \mathcal{F} \) and \( C \in \mathcal{F}^* \), and \(A \subseteq D \subseteq X \), then \( B \cup D \in \mathcal{F} \) (since \( \mathcal{F} \) is a stack) and \( C \cup D \in \mathcal{F}^* \) (since, by statement (a), \( \mathcal{F}^* \) is a stack) implies \( D = A \cup D = (B \cap C) \cup D = (B \cup D) \cap (C \cup D) \in \mathcal{G} \).
     This shows that \( \mathcal{G} \) is a stack.
     
     To see \( \mathcal{G} \) is a grill let \( A_1 \cup A_2 \in \mathcal{G} \) and pick \( B \in \mathcal{F} \) and \( C \in \mathcal{F}^* \) with \( A_1 \cup A_2 = B \cap C \).
     Then \( A_2 \setminus A_1 = (X \setminus A_1) \cap (A_1 \cup A_2) = (B \setminus A_1) \cap C \).
     If \( B \setminus A_1 \in \mathcal{F} \), then \( A_2 \setminus A_1 \in \mathcal{G} \) and so \( A_2 \in \mathcal{G} \).
     If \( B \setminus A_1 \not\in \mathcal{F} \), then \( (X \setminus B) \cup A_1 \in \mathcal{F}^*\) and so \( B \cap  A_1 = B \cap \bigl((X \setminus B) \cup A_1\bigr) \in \mathcal{G} \) and hence \( A_1 \in \mathcal{G} \).
\end{proof}

For us \cref{prop:mesh-operator}(f)~and~(g) are of fundamental importance to the ultrafilter description of the Stone--\v{C}ech compactification and its connections to Ramsey Theory. 
For instance see \cite[Theorems~3.11 and 5.7]{Hindman:2012tq}.
The earliest reference we've found for \cref{prop:mesh-operator}(h) is \cite[Proposition~2.1(e)]{Akin1997} (in Akin's formulation it asserts that the set \( \{\, B \cap C : B \in \mathcal{F} \mbox{ and } C \in \mathcal{F}^* \,\}^* \) is a filter).
Our proof follows a blogpost of Moreira \cite[Proposition~3]{Moreira:2016uz}.

One application of \cref{prop:mesh-operator}(h) is that it implies \( \mathsf{PS} \) is a grill (this implication is also noted by Moreira \cite[Corollary~4]{Moreira:2016uz}):
\begin{corollary}%[\textit{Piecewise syndetic sets are partition regular}]
  \label{corollary:brown-lemma}
  Let \( (S, \cdot) \) be a semigroup.
  Then the collection of all piecewise syndetic sets \( \mathsf{PS} \) is a grill on \( S \).
\end{corollary}
\begin{proof}

	By definition \( \mathsf{PS} = \{ B \cap C : B \in \mathsf{Syn} \mbox{ and } C \in \mathsf{Thick} \} \), and so by \cref{prop:mesh-operator}(h) and the duality \( \mathsf{Syn} = \mathsf{Thick}^* \) we have \( \mathsf{PS} \) is a grill.
\end{proof}

\subsection*{Brief review of algebraic structure of the Stone--\v{C}ech compactification}
\label{section:algebraic-review}
We'll end this section by fixing some additional notation
and giving a brief review of the algebraic structure of the Stone--\v{C}ech compactification of a discrete semigroup.

A standing convention throughout this paper is \( S \) will denote an infinite discrete semigroup with \( \cdot \) as its binary operation.
We take \( \beta S \) to be the collection of all ultrafilters on \( S \).
We identify points of \( S \) with the principal ultrafilters in \( \beta S \), pretending that \( S \subseteq \beta S \).
Given \( A \subseteq S \) we put \( \overline{A} = \{ p \in \beta S : A \in p \} \).
Then the collection \( \{\overline{A} : A \subseteq S \} \) is a basis for a compact Hausdorff topology on \( \beta S \) and \( c\ell_{\beta S}(A) = \overline{A} \).
This topology is the Stone--\v{C}ech compactification of \( S \).

If \( \mathcal{F} \) is a filter on \( S \), we put \( \overline{\mathcal{F}} = \{p \in \beta S : \mathcal{F} \subseteq p \} \).
Then \( \overline{\mathcal{F}} \) is a nonempty closed subset of \( \beta S \);
conversely,
any nonempty closed subset of \( \beta S \) is uniquely generated by some filter.
The empty subset of \( \beta S \) is generated by the so-called ``improper filter'' \( \mathcal{P}(S) \).
(More precisely,
if \( C \subseteq \beta S \) is a closed subset,
then \( \bigcap \{ p \in \beta S : p \in C \} \) is the filter that generates it.)

The proofs of these assertions can be found in \cite[Sections~3.2 and 3.3]{Hindman:2012tq}.

The semigroup operation on \( S \) can be extended to a semigroup operation on \( \beta S \) \cite[Theorem~4.1]{Hindman:2012tq} such that for every \( p, q \in \beta S \) we have \( A \in p \cdot q \) if and only if \( \{x \in S : x^{-1}A \in q \} \in p \) \cite[Theorem~4.12]{Hindman:2012tq}.
This extension makes \( (\beta S, \cdot) \) a compact right topological semigroup. 
Right topological means that for every \( q \in \beta S \) the map \( p \mapsto p \cdot q \) for all \( p \in \beta S \) is continuous. 

\textbf{Important Note:}
Several of our references, \cite{Berglund1984, davenport:1987, Davenport:1990wq}, take \( \beta S \) to be left topological.
We take our algebraic structure on \( \beta S \) to be right topological.
In all the cases we cite, the appropriate left-right switches of their proofs and statements give the corresponding right topological version.

Compact Hausdorff right topological semigroups contain significant algebraic structure.
For instance, \( (\beta S, \cdot) \) has idempotent elements (usually many of them) \cite[Theorem~2.5]{Hindman:2012tq} and a smallest ideal \( K(\beta S) \) which is the union of all minimal left ideals of \( \beta S \) and also the union of all minimal right ideals of \( \beta S \) \cite[Theorem~2.8]{Hindman:2012tq}.

The characterization \( A \in p \cdot q \iff \{\, x \in S : x^{-1}A \in q \,\} \in p \) can be taken as the definition for \( \mathcal{F} \cdot \mathcal{G} \) when both \( \mathcal{F} \) and \( \mathcal{G} \) are stacks.
That is, we can define \( \mathcal{F} \cdot \mathcal{G} = \bigl\{\, A \subseteq S : \{\, x \in S : x^{-1}A \in \mathcal{G} \,\} \in \mathcal{F} \,\bigr\} \).
Berglund and Hindman proved \cite[Lemma~5.15]{Berglund1984} that when both \( \mathcal{F} \) and \( \mathcal{G} \) are filters, this product is also a filter and associative.
We simply note that their proof (with appropriate left-right switches) also shows, if both \( \mathcal{F} \) and \( \mathcal{G} \) are stacks, then this product is also a stack and is associative. 

\section{Relative notions of syndetic and thick sets}
\label{section:relative-syndetic-and-thick}
The goal of this section is to demonstrate how the algebraic structure of \( \beta S \) can be used to concisely characterize many instances of relative notions of syndetic and thick sets.

\begin{definition}
	\label{definition:relative-syndetic-and-thick}
	Let \( A \subseteq S \) and let \( \mathcal{F} \) and \( \mathcal{G} \) both be stacks on \( S \).
  
	\begin{itemize}   
    \item[(a)]
      \( A \) is \define{\( (\mathcal{F}, \mathcal{G}) \)-syndetic}
      if and only if
      for every \( B \in \mathcal{F} \) there exists \( H \in \mathcal{P}_f(B) \) such that \( \bigcup_{h \in H} h^{-1}A \in \mathcal{G} \).
      
	 \item[(b)] 
      \( A \) is \define{\( (\mathcal{F}, \mathcal{G}) \)-thick}
      if and only if
      there exists \( B \in \mathcal{F} \) such that for every \( H \in \mathcal{P}_f(B) \) we have \( \bigcap_{h \in H} h^{-1}A \in \mathcal{G}^* \).
      %there is an \( X \in \mesh{\mathcal{G}} \) such that \( F \cdot X \subseteq A \).
      % In terms of the Cayley table of S, this says you can take an F-large number of rows and for any finite collection of these rows you can find a G-sieve-large number of columns whose entries are contained in your (F,G)-thick set.
      
      \item[(c)] We also define the following two collections:
	      \begin{itemize}
    	  	\item[(i)]  \( \mathsf{Syn}(\mathcal{F}, \mathcal{G}) = \{A \subseteq S : A \mbox{ is \( (\mathcal{F}, \mathcal{G}) \)-syndetic} \} \).

    	  	\item[(ii)]  \( \mathsf{Thick}(\mathcal{F}, \mathcal{G}) = \{A \subseteq S : A \mbox{ is \( (\mathcal{F}, \mathcal{G}) \)-thick} \} \).
	      \end{itemize}
 	 \end{itemize}
\end{definition}

The notation and name of ``\( (\mathcal{F}, \mathcal{G}) \)-syndetic'' is due to Shuungula, Zelenyuk, and Zelenyuk \cite[third paragraph of Section~2]{Shuungula:2009ty}, and our approach to these notions are influenced by the results and methods in their paper (for instance, see \cref{remark:choice-of-notation}).
They used this notion, when \( \mathcal{F} \) and \( \mathcal{G} \) are both filters, to prove a simple characterization of the smallest ideal in a closed subsemigroup of \( \beta S \) \cite[Theorem~2.2]{Shuungula:2009ty} analogous to a characterization of the smallest ideal of \( \beta S \) \cite[Theorem~4.39]{Hindman:2012tq}.
Davenport obtained a similar characterization (after a bit of rewriting) earlier \cite[Theorem~3.4]{Davenport:1990wq}.

% For notational convenience, following \cite{Shuungula:2009ty}, we write \( \mathcal{F} \)-syndetic and \( \mathcal{F} \)-thick when we mean \( (\mathcal{F}, \mathcal{F}) \)-syndetic and \( (\mathcal{F}, \mathcal{F}) \)-thick, respectively.
Observe \( \mathsf{Syn} = \mathsf{Syn}(\{S\}, \{S\}) \) and \( \mathsf{Thick} = \mathsf{Thick}(\{S\}, \{S\}) \).

The notions of \( (\mathcal{F}, \mathcal{F}) \)-syndetic and \( (\mathcal{F}, \mathcal{F}) \)-thick, again when \( \mathcal{F} \) is a filter, also appears in a note of Protasov and Sloboadianiuk \cite{Protasov:2015uz} as ``\( \tau \)-large'' and ``\( \tau \)-thick'', respectively. 
Among other things, they also characterize the smallest ideal of a closed subsemigroup of \( \beta S \) \cite[Theorem~3.1]{Protasov:2015uz}.

Additionally, Blass and Di Nasso \cite{Blass2016} also studied the notion of ``finite embeddability'' and its connection to the algebraic structure of the Stone--\v{C}ech compactification.
In our terminology and notation, given two subsets \( A \) and \( B \) of the nonnegative integers, \( \mathbb{Z}_{\ge 0} \), they defined \( A \) is \define{finitely embeddable} in \( B \) if and only if \( B \) is \( (\mathcal{F}, \{\mathbb{Z}_{\geq 0}\}) \)-thick where the first component is the principal filter \( \mathcal{F} = \{ C \subseteq \mathbb{Z}_{\ge 0} : A \subseteq C \} \) generated by \( A \).
Baglini further studied \cite{LuperiBaglini2014} and extended \cite{LuperiBaglini2016} the concept of finite embeddability.

Similar to syndetic and thick, it's easy to observe these relative notions share a duality:
\begin{proposition}%[\textit{Duality principle between relative notions of syndetic and thick}]
\label{proposition:duality-principle}
  Let \( \mathcal{F} \) and \( \mathcal{G} \) both be stacks on a semigroup \( S \).
  
  \begin{itemize}
	\item[(a)] \( \mathsf{Syn}(\mathcal{F}, \mathcal{G}) = \mathsf{Thick}(\mathcal{F}, \mathcal{G})^* \).
	\item[(b)] \( \mathsf{Syn}(\mathcal{F}, \mathcal{G}) = \{ A \subseteq S : (\forall B \in \mathsf{Thick}(\mathcal{F}, \mathcal{G})) \; A \cap B \ne \emptyset \} \).
  \end{itemize}
\end{proposition}
\begin{proof}
  \textbf{(a)} 
  The justification for statement (a) is ultimately an easy application of De Morgan's laws:
   \begin{align*}
	A \in \mathsf{Syn}(\mathcal{F}, \mathcal{G}) &\iff (\forall B \in \mathcal{F})\bigl(\exists H \in \mathcal{P}_f(B)\bigr)\,\; \bigcup_{h \in H} h^{-1}A \in \mathcal{G} \\
	&\iff \neg (\exists B \in \mathcal{F})\bigl(\forall H \in \mathcal{P}_f(B)\bigr)\; \bigcap_{h \in H} h^{-1}(S \setminus A) \in \mathcal{G}^* \\
	&\iff S \setminus A \not\in \mathsf{Thick}(\mathcal{F}, \mathcal{G}) \iff A \in \mathsf{Thick}(\mathcal{F}, \mathcal{G})^*.
 \end{align*}

  \textbf{(b)} 
  Statement (b) then follows immediately from \cref{prop:mesh-operator}(d).
\end{proof}

\begin{remark}
	\label{remark:choice-of-notation}
	The reader may wonder why we refer to the mesh of a collection \( \mathcal{G}^* \) in the definition of \( (\mathcal{F}, \mathcal{G}) \)-thick and \textit{not indicate this} directly in the notation, perhaps by calling such sets ``\( (\mathcal{F}, \mathcal{G}^*) \)-thick'' instead.
	Our choice was made so we can follow the notation as used by Shuungula, Zelenyuk, and Zelenyuk \cite{Shuungula:2009ty} and maintain the duality principle between the notions of \( (\mathcal{F}, \mathcal{G}) \)-thick and \( (\mathcal{F}, \mathcal{G}) \)-syndetic (see \cref{proposition:duality-principle} above).
	%Our choice was also determined that also corresponds to a nice algebraic characterizations for certain of these notions (see Lemma~\ref{lemma:algebraic-characterizations}(a)).  
\end{remark}

\begin{example}%[\textit{Relative thick does not imply thick}]
  \label{example:relative-thick-does-not-imply-thick}
  As the reader probably suspects, our relative notion of thickness does not imply the usual notion of thickness even in \( (\mathbb{N}, +) \).
  Consider the principal filter generated by the even positive integers \( \mathcal{F} = \{A \subseteq \mathbb{N} : 2\mathbb{N} \subseteq A \} \).
  Observe that \( A \in \mathsf{Thick}(\mathcal{F}, \{\mathbb{N}\}) \) if and only if for all \( F \in \mathcal{P}_f(\mathbb{N}) \) there exists \( x \in \mathbb{N} \) such that \( 2 F + x \subseteq A \).
  In particular, the set of even positive integers is \( (\mathcal{F}, \{\mathbb{N}\}) \)-thick but is not thick.
  
  Hence \( (\mathcal{F}, \{\mathbb{N}\}) \)-thick is a weaker notion than thick (since it allows more sets to be ``thick''), that is, \( \mathsf{Thick} \subsetneq \mathsf{Thick}(\mathcal{F}, \{\mathbb{N}\}) \).
  Also, by Proposition~\ref{proposition:duality-principle}, this is equivalent to writing \( \mathsf{Syn}(\mathcal{F}, \{\mathbb{N}\}) \subsetneq \mathsf{Syn} \), that is, \( (\mathcal{F}, \{\mathbb{N}\}) \)-syndetic is a stronger notion than syndetic (since it allows fewer sets to be ``syndetic''). 
\end{example}

As \cref{example:relative-thick-does-not-imply-thick} indicates, we also have the following order-reversing [order-preserving] implications for relative syndetic and thick sets:
\begin{proposition}
    \label{proposition:ordering}
    Let \( S \) be a semigroup and let \( \mathcal{F} \), \( \mathcal{F}_1 \),  \( \mathcal{F}_2 \), \( \mathcal{G} \), \( \mathcal{G}_1 \), and \( \mathcal{G}_2 \) all be stacks on \( S \).
    \begin{itemize}
        \item[(a)] If \( \mathcal{F}_1 \subseteq \mathcal{F}_2 \), then \(\mathsf{Syn}(\mathcal{F}_2, \mathcal{G}) \subseteq \mathsf{Syn}(\mathcal{F}_1, \mathcal{G}) \). 
        
        \item[(b)] If \( \mathcal{G}_1 \subseteq \mathcal{G}_2 \), then \( \mathsf{Syn}(\mathcal{F}, \mathcal{G}_1) \subseteq \mathsf{Syn}(\mathcal{F}, \mathcal{G}_2) \).
        
        \item[(a\(^ \prime \))] If \( \mathcal{F}_1 \subseteq \mathcal{F}_2 \), then \(\mathsf{Thick}(\mathcal{F}_1, \mathcal{G}) \subseteq \mathsf{Thick}(\mathcal{F}_2, \mathcal{G}) \). 
        
        \item[(b\(^ \prime \))]  If \( \mathcal{G}_1 \subseteq \mathcal{G}_2 \), then \( \mathsf{Thick}(\mathcal{F}, \mathcal{G}_2) \subseteq \mathsf{Thick}(\mathcal{F}, \mathcal{G}_1)  \).      
    \end{itemize}
\end{proposition}
\begin{proof}
    By \cref{proposition:duality-principle}(c), statement (a) is equivalent to statement (a\(^ \prime \)) and statement (b) is equivalent to statement (b\(^ \prime \)).
    The proofs of statements (a) and (b) are each straightforward one-line verifications.

    \textbf{(a)} Let \( A \in \mathsf{Syn}(\mathcal{F}_2, \mathcal{G}) \) and let \( B \in \mathcal{F}_1 \subseteq \mathcal{F}_2\). 
    Pick \( H \in \mathcal{P}_f(B) \) such that \( \bigcup_{h \in H} h^{-1}A \in \mathcal{G} \) as guaranteed by \( A \). 
    
    \textbf{(b)} Let \( A \in \mathsf{Syn}(\mathcal{F}, \mathcal{G}_1) \) and let \( B \in \mathcal{F} \). 
    Pick \( H \in \mathcal{P}_f(B) \) such that \( \bigcup_{h \in H} h^{-1}A \in \mathcal{G}_1 \subseteq \mathcal{G}_2 \).
\end{proof}

We also note that that relative syndetic and thick collections \emph{do not necessarily preserve strict inclusion} among either the left or right components:

\begin{example}
    \label{example:strict-inclusion-not-preserve}
    In \( (\mathbb{N}, +) \) we have \( \mathsf{Thick}(\{ \mathbb{N} \}, \{ \mathbb{N} \}) = \mathsf{Thick}(\mathcal{C}, \{ \mathbb{N} \}) = \mathsf{Thick}(\{ \mathbb{N} \}, \mathcal{C}) = \mathsf{Thick}(\mathcal{C}, \mathcal{C}) \), where \( \mathcal{C} = \{A \subseteq \mathbb{N} : \mathbb{N} \setminus A \text{ is finite} \} \) is the cofinite filter on \( \mathbb{N} \).
    (This can either be checked directly from the definitions or noting that \( \overline{\mathcal{C}} \) (that is, the collection of all non-principal ultrafilters on \( \mathbb{N} \)) is a closed ideal of \( \beta \mathbb{N} \) and applying \cref{lemma:algebraic-characterizations}(c) and \cref{theorem:algebraic-characterizations}.)
\end{example}

One advantage in following Shuungula, Zelenyuk, and Zelenyuk's notation is that it allows us to ``compose'' these notions to produce (possibly new) relative notions of size.
To get a sense of what we mean by composing these notions consider \( \mathsf{Thick}(\mathcal{F}, \mathcal{G}) \) where \( \mathcal{F} \) and \( \mathcal{G} \) are in \( \{\, \{S\}, \mathsf{Syn}, \mathsf{Thick} \,\} \).
In this case we have at most \textit{nine different notions of being (relatively) thick}.
Some of these notions of relative thick are known.
For instance, we have the following characterization for the collection of all piecewise syndetic sets \( \mathsf{PS} \) as a composition (starting with a well known characterization \( \mathsf{PS} \)):
\begin{proposition}%[\textit{Characterization of piecewise syndetic}]
  \label{proposition:characterization-of-PS-as-broken-syndetic}
  Let \( S \) be a semigroup.
  \begin{itemize}
    \item[(a)] \( \mathsf{PS} = \mathsf{Syn}(\{S\}, \mathsf{Thick}) \).
  	\item[(b)] \( \mathsf{PS} = \mathsf{Thick}(\mathsf{Syn}, \mathsf{PS}^*) \).
  	\item[(c)] \( \mathsf{PS}^* = \mathsf{Thick}(\{S\}, \mathsf{Thick}) \).
  	\item[(d)] \( \mathsf{PS} = \mathsf{Thick}(\mathsf{Syn}, \mathsf{Thick}(\{S\}, \mathsf{Thick})) \).
  \end{itemize}
\end{proposition}
\begin{proof}
    \textbf{(a)}
     It is a result of Bergelson, Hindman, and McCutcheon \cite[Theorem~2.4(d)]{Bergelson1998} (or also see \cite[Theorem~4.49]{Hindman:2012tq}) that \( A \in \mathsf{PS} \) if and only if there exists \( H \in \mathcal{P}_f(S) \) such that \( \bigcup_{h \in H} h^{-1}A \) is thick.
%    This is just \cref{definition:syndetic-and-thick}(c).
    
    \textbf{(b)}
    \( (\subseteq) \)
    Let \( A \in \mathsf{PS} \).
    By \cite[Theorem~4.40]{Hindman:2012tq}, we have that \( A \) is piecewise syndetic if and only if \( K(\beta S) \cap \overline{A} \ne \emptyset \).
    Pick \( q \in K(\beta S) \cap \overline{A} \) as guaranteed, and by \cite[Theorem~4.39]{Hindman:2012tq}, we have \(\{x \in S : x^{-1}A \in q \} \) is syndetic.
    Putting \( B = \{x \in S : x^{-1}A \in q \} \) we see that for every \( H \in \mathcal{P}_f(B) \) we have that \( \bigcap_{x \in H} x^{-1}A \in q \subseteq \mathsf{PS} \), since \( q \) is an ultrafilter in \( K(\beta S) \).
    Hence \( A \in \mathsf{Thick}(\mathsf{Syn}, \mathsf{PS}^*) \).

      \( (\supseteq) \)
      We'll provide a combinatorial proof of this direction. 
    Let \( A \in \mathsf{Thick}(\mathsf{Syn}, \mathsf{PS}^*) \) and pick \( B \in \mathsf{Syn} \) as guaranteed for \( A \).
  Pick \( H \in \mathcal{P}_f(S) \) such that \( \bigcup_{h \in H} h^{-1}B = S \).
  We claim that \( \bigcup_{h \in H} h^{-1}A \) is thick.
  
  Let \( F \in \mathcal{P}_f(S) \) and for each \( f \in F \) pick \( h_f \in H \) such that \( h_f \cdot f \in B \).
  Then \( \{ h_f \cdot f : f \in F \} \) is a finite nonempty subset of \( B \). 
  We can pick \( X \in \mathsf{PS} \) such that \( \{ h_f \cdot f : f \in F \} \cdot X \subseteq A \). 
  It follows that \( F \cdot X \subseteq \bigcup_{h \in H} h^{-1}A \) and so \( A \in \mathsf{PS} \).    
  
  This completes the proof of statement (b).

    \textbf{(c)} 
    Applying \cref{proposition:duality-principle} to statement (a) yields  \( \mathsf{PS}^* = \mathsf{Thick}(\{S\}, \mathsf{Thick}) \).
	
    \textbf{(d)} 
    Combining statements (b) and (c) immediately proves this equivalence.
\end{proof}

A combinatorial proof of a special case (in our notation, \( \mathsf{PS} = \mathsf{Thick}(\mathsf{Syn}, \{S\}) \)) of statement (a), in the context of \( (\mathbb{N}, +) \), is given in \cite[Lemma~5]{Moreira:2016uz}.
The characterization for \( \mathsf{PS}^* \), at least in the context of a group, appears at least as early as \cite[Definition~1.4]{Bergelson1988} under the name ``permanently syndetic''.
It appears that statement (d) is a new characterization of piecewise syndetic sets.
(We wouldn't be surprised if this latter characterization or its dual has appeared previously in the literature or folklore.)

As this result illustrates, it is not immediately clear how all of these notions are related.
Moreover, developing an algebraic theory of their classification or characterization seems it could be a challenging but important project.
(Important because the algebraic characterizations of notions of size indicate the underlying structure of ``large'' sets.)
Therefore we propose the following characterization problem:
\begin{problem}[Characterization problem on composed notions of size]
    Develop an algebraic characterization of composed notions of size.
    Given stacks \( \mathcal{F} \) and \( \mathcal{G} \) is it possible to develop a systemic  characterization of \( \mathsf{Syn}(\mathcal{F}, \mathcal{G}) \) and \(\mathsf{Thick}(\mathcal{F}, \mathcal{G}) \) using the algebraic structure of \( \beta S \)?
\end{problem}

In \cref{lemma:algebraic-characterizations} we solve special instances of the above ``characterization problem'' in terms of closed subsets of \( \beta S \), but, as \cref{proposition:characterization-of-PS-as-broken-syndetic} indicates, we believe more can be said on this point.
\begin{lemma}%[\textit{Algebraic characterizations}]
  \label{lemma:algebraic-characterizations}
  Let \( \mathcal{F} \) and \( \mathcal{G} \) both be filters on a semigroup \( S \).
  \begin{itemize}
    \item[(a)]
      \( \mathsf{Thick}(\mathcal{F}, \mathcal{G}) = \{ A \subseteq S : \mbox{there exists } q \in \overline{\mathcal{G}} \mbox{ such that } \overline{\mathcal{F}} \cdot q \subseteq \overline{A} \} \)
    
    \item[(b)]
      \( \mathsf{Thick}(\mathcal{F}, \mathcal{G}^*) = \{ A \subseteq S : \overline{\mathcal{F} \cdot \mathcal{G}} \subseteq \overline{A} \} \)

    \item[(c)]
      \( \mathsf{Thick}(\mathcal{F}^*, \mathcal{G}) = \{ A \subseteq S : \overline{\mathcal{F}} \cdot \overline{\mathcal{G}} \cap \overline{A} \ne \emptyset \} \)      

    \item[(d)]
      \( \mathsf{Thick}(\mathcal{F}^*, \mathcal{G}^*) = \{ A \subseteq S : \mbox{there exists } p \in \overline{\mathcal{F}} \mbox{ such that } \overline{p \cdot \mathcal{G}} \subseteq \overline{A} \} \)

    \item[(a\(^ \prime \))]
      \( \mathsf{Syn}(\mathcal{F}, \mathcal{G}) = \{ A \subseteq S : \mbox{for every } q \in \overline{\mathcal{G}} \mbox{ we have } \overline{\mathcal{F}} \cdot q \cap \overline{A} \ne \emptyset \} \)
%    \item[(e)]
%      \( A \) is \( (\mathcal{F}, \mathcal{G}) \)-syndetic if and only if for every \( q \in \overline{\mathcal{G}} \) such that \( \overline{A} \cap \overline{\mathcal{F}} \cdot q \ne \emptyset \).

    \item[(b\(^ \prime \))]
      \( \mathsf{Syn}(\mathcal{F}, \mathcal{G}^*) = \{ A \subseteq S : \overline{\mathcal{F} \cdot \mathcal{G}} \cap \overline{A} \ne \emptyset \} \)
%    \item[(f)]
%      \( A \) is \( (\mathcal{F}^{*}, \mathcal{G}) \)-syndetic if and only if \( \overline{\mathcal{F}} \cdot \overline{\mathcal{G}} \subseteq \overline{A} \).

    \item[(c\(^ \prime \))]
      \( \mathsf{Syn}(\mathcal{F}^*, \mathcal{G}) = \{ A \subseteq S : \overline{\mathcal{F}} \cdot \overline{\mathcal{G}} \subseteq \overline{A} \} \).
%    \item[(g)]
%      \( A \) is \( (\mathcal{F}, \mathcal{G}^{*}) \)-syndetic if and only if \( \overline{A} \cap \overline{\mathcal{F} \cdot \mathcal{G}} \ne \emptyset \).

    \item[(d\(^ \prime \))]
      \( \mathsf{Syn}(\mathcal{F}^*, \mathcal{G}^*) = \{ A \subseteq S : \mbox{for every } p \in \overline{\mathcal{F}} \mbox{ we have } \overline{p \cdot \mathcal{G}} \cap \overline{A} \ne \emptyset \} \).
%    \item[(h)]
%      \( A \) is \( (\mathcal{F}^{*}, \mathcal{G}^{*}) \)-syndetic if and only if for every \( p \in \overline{\mathcal{F}} \) we have \( \overline{A} \cap \overline{p \cdot \mathcal{G}} \ne \emptyset \).
  \end{itemize}
\end{lemma}
\begin{proof}
    By \cref{proposition:duality-principle} statements (a), (b), (c), and (d) are equivalent to the corresponding statements (a\(^ \prime \)), (b\(^ \prime \)), (c\(^ \prime \)), and (d\(^ \prime \)), respectively.
    
    \textbf{(a)}
     \( (\subseteq) \)
  Let \( A \in \mathsf{Thick}(\mathcal{F}, \mathcal{G}) \)  and pick \( B \in \mathcal{F} \) as guaranteed for \( A \).
  By \cite[Theorem~3.11]{Hindman:2012tq}, we can pick \( q \in \overline{\mathcal{G}} \) with \( \{x^{-1}A : x \in B \} \subseteq q \).
  If \( p \in \overline{\mathcal{F}} \), then \( B \in p \) (since \( \mathcal{F} \subseteq p \)).
  Now \( B \subseteq \{x \in S : x^{-1}A \in q\} \) implies \( \{x \in S : x^{-1}A \in q\} \in p \), that is, \( A \in p \cdot q \).
  Hence \( \overline{\mathcal{F}}\cdot q \subseteq \overline{A} \).
  
  \( (\supseteq) \)
  Now assume we have \( A \subseteq S \) such that there exists \( q \in \overline{\mathcal{G}} \) with \( \overline{\mathcal{F}}\cdot q \subseteq \overline{A} \).
  Then \( A \in p \cdot q \) for every \( p \in \overline{\mathcal{F}} \), that is, \( \{ x \in S : x^{-1}A \in q \} \in p \) for every \( p \in \overline{\mathcal{F}} \).
  Therefore \( \{ x \in S : x^{-1}A \in q \} \in \mathcal{F} \), and if we put \( B = \{ x \in S : x^{-1}A \in q \} \) and let \( H \in \mathcal{P}_f(B) \) we have \( \bigcap_{h \in H} h^{-1}A \in q \subseteq \mathcal{G}^* \).
  Hence \( A \in \mathsf{Thick}(\mathcal{F}, \mathcal{G}) \).

  This completes the proof of statement (a).

    \textbf{(b)}
    Recall, by \cite[Lemma~5.15]{Berglund1984}, we have \( \mathcal{F} \cdot \mathcal{G} \) is a filter on \( S \).
    
     \( (\subseteq) \)
  Let \( A \in \mathsf{Thick}(\mathcal{F}, \mathcal{G}^*) \) and pick \( B \in \mathcal{F} \) as guaranteed for \( A \), that is, for all \( H \in \mathcal{P}_f(B) \) we have \( \bigcap_{h \in H} h^{-1}A \in \mathcal{G} \).
  Then  \( B \subseteq \{x \in S : x^{-1}A \in \mathcal{G} \} \) and \( B \in \mathcal{F} \) implies \( \{x \in S : x^{-1}A \in \mathcal{G} \} \in \mathcal{F} \), that is, \( A \in \mathcal{F} \cdot \mathcal{G} \).
  Hence \( \overline{\mathcal{F} \cdot \mathcal{G}} \subseteq \overline{A} \).
  
     \( (\supseteq) \)
  Now assume we have \( A \subseteq S \) such that \( \overline{\mathcal{F} \cdot \mathcal{G}} \subseteq \overline{A} \).
  Then \( A \in \mathcal{F} \cdot \mathcal{G} \), that is, \( \{ x \in S : x^{-1}A \in \mathcal{G} \} \in \mathcal{F} \).
  Put \( B = \{ x \in S : x^{-1}A \in \mathcal{G} \} \) and let \( H \in \mathcal{P}_f(B) \).
  We have \( \bigcap_{h \in H} h^{-1}A \in \mathcal{G} \) and hence \( A \in \mathsf{Thick}(\mathcal{F}, \mathcal{G}^*) \).
  
  This completes the proof of statement (b).
  
  \textbf{(c)}
  \( (\subseteq) \)
  Let \( A \in \mathsf{Thick}(\mathcal{F}^*, \mathcal{G}) \) and pick \( B \in \mathcal{F}^* \) as guaranteed for \( A \).
  Similar to our proof for statement (a), by \cite[Theorem~3.11]{Hindman:2012tq}, we can pick \( q \in \overline{G} \) such that \( \{x^{-1}A : x \in B\} \subseteq q \).
  Pick \( p \in \overline{\mathcal{F}} \) with \( B \in p \).
  Then \( B \subseteq \{ x \in S : x^{-1}A \in q \} \) and \( B \in p \) implies \( \{x \in S : x^{-1}A \in q\} \in p \), that is, \( A \in p \cdot q \).

  \( (\supseteq) \)
  Now assume we have \( A \subseteq S \) such that that \( \overline{\mathcal{F}} \cdot \overline{\mathcal{G}} \cap \overline{A} \ne \emptyset \).
  Pick \( p \in \overline{\mathcal{F}} \) and \( q \in \overline{\mathcal{G}} \) with \( p \cdot q \in \overline{A} \), that is, \( \{ x \in S : x^{-1}A \in q \} \in p \).
  Put \( B = \{ x \in S : x^{-1}A \in q \} \) and let \( H \in \mathcal{P}_f(B) \).
  Then \( \bigcap_{h \in H} h^{-1}A \in q \subseteq \mathcal{G}^* \).
  Hence \( A \) is \( (\mathcal{F}^{*}, \mathcal{G}) \)-thick.

  This completes the proof of statement (c).

  \textbf{(d)}
  \( (\subseteq) \)
  Let \( A \in \mathsf{Thick}(\mathcal{F}^{*}, \mathcal{G}^{*}) \) and pick \( B \in \mathcal{F}^{*} \) as guaranteed for \( A \).
  Then it follows that \( B \subseteq \{x \in S : x^{-1}A \in \mathcal{G} \} \), since \( \mathcal{G} \) is a filter.
  Pick \( p \in \overline{\mathcal{F}} \) with \( B \in p \).
  Then \( A \in p \cdot \mathcal{G} \), that is, \( \overline{p \cdot \mathcal{G}} \subseteq \overline{A} \).

  \( (\supseteq) \)    
  Now assume we have \( A \subseteq S \) such that \( \overline{p \cdot \mathcal{G}} \subseteq \overline{A} \) for some \( p \in \overline{\mathcal{F}} \).
  Then \( \{x \in S : x^{-1}A \in \mathcal{G} \} \in p \).
  Put \( B = \{x \in S : x^{-1}A \in \mathcal{G} \} \) (so \( B \in p \subseteq \mathcal{F}^{*} \)) and let \( H \in \mathcal{P}_f(B) \).
  Then \( \bigcap_{h \in H} h^{-1}A \in \mathcal{G} \) and hence \( A \in \mathsf{Thick}(\mathcal{F}^*, \mathcal{G}^*) \).

  This completes the proof of statement (d).
\end{proof}

When \( \mathcal{F} = \mathcal{G} = \{S\} \), \cref{lemma:algebraic-characterizations}(a) and (a\( ^\prime \)) imply the characterizations for thick and syndetic given earlier in \cref{theorem:algebraic-characterizations}.
\cref{lemma:algebraic-characterizations} also generalizes a characterization for \( \mathsf{Syn}(\mathcal{F}, \mathcal{G}) \), when both \( \mathcal{F} \) and \( \mathcal{G} \) are filters and \( \mathcal{F} \subseteq \mathcal{G} \), proved by Shuungula, Zelenyuk, and Zelenyuk \cite[Lemma~2.1]{Shuungula:2009ty}.
Additionally, it generalizes the characterizations for \( \tau \)-large and \( \tau \)-thick in \cite[Theorems~2.1 and 2.2]{Protasov:2015uz}, and generalizes Blass and Di Nasso's formulations of finite embeddability \cite[Theorem~4]{Blass2016}.

The remaining characterizations appear to be mostly new, but we suspect that they appear both explicitly and (in a sense, necessarily) implicitly in much of the literature on algebra in \( \beta S \).

If \( p, q \in \beta S \), then \( A \) is \( (p, q) \)-thick if and only if \( A \) is \( (p, q) \)-syndetic if and only if \( A \in p \cdot q \), that is, all of these notions collapse to a product of two ultrafilters when we substitute \( p, q \) for \( \mathcal{F}, \mathcal{G} \).
Hence \cref{lemma:algebraic-characterizations} can be thought of as a generalization of \cite[Theorem~4.12]{Hindman:2012tq}, but, of course, in the proof of our lemma we used this latter result implicitly throughout.

One convenient aspect of \cref{lemma:algebraic-characterizations} is that we can easily create new filters and hence describe the corresponding closed subsets of  \( \beta S \) in a ``combinatorial way''.
%Moreover, we can apply \cref{proposition:ordering} to note when certain collections are contained in other collections.
We note two important examples of this in the following theorem:
\begin{theorem}%[\textit{Two important filters}]
  \label{theorem:important-filters}
  Let  \( \mathcal{F} \) and  \( \mathcal{G} \) both be filters on a semigroup  \( S \).
  \begin{itemize}
    \item[(a)]  \( \mathsf{Syn}(\mathcal{F}^*, \mathcal{G}) \) is a filter on  \( S \) and  \( \overline{\mathsf{Syn}(\mathcal{F}^*, \mathcal{G})} = c\ell(\overline{\mathcal{F}} \cdot \overline{\mathcal{G}}) \).
    \item[(b)]  \( \mathsf{Thick}(\mathcal{F}, \mathcal{G}^*) \) is a filter on  \( S \) and  \( \overline{\mathsf{Thick}(\mathcal{F}, \mathcal{G}^*)} = \overline{\mathcal{F} \cdot \mathcal{G}} \).
  \end{itemize}
\end{theorem}
\begin{proof}
    \textbf{(a)}
    From \cref{lemma:algebraic-characterizations}(c\(^ \prime \)) we have \( \mathsf{Syn}(\mathcal{F}^*, \mathcal{G}) = \{ A \subseteq S : \overline{\mathcal{F}} \cdot \overline{\mathcal{G}} \subseteq \overline{A} \} \).
    With this characterization we can easily verify that \( \mathsf{Syn}(\mathcal{F}^*, \mathcal{G}) \) is a filter on \( S \).
    We have \( S \in \mathsf{Syn}(\mathcal{F}^*, \mathcal{G}) \) (since \( \overline{\mathcal{F}} \cdot \overline{\mathcal{G}} \subseteq \beta S = \overline{S} \)) and \( \emptyset \not\in \mathsf{Syn}(\mathcal{F}^*, \mathcal{G}) \) (since \(\overline{\emptyset} = \emptyset \)).
    If \( A \in \mathsf{Syn}(\mathcal{F}^*, \mathcal{G}) \) and \( A \subseteq B \subseteq S \), then \( \overline{\mathcal{F}} \cdot \overline{\mathcal{G}} \subseteq \overline{A} \subseteq \overline{B} \) and hence \( B \in \mathsf{Syn}(\mathcal{F}^*, \mathcal{G}) \).
    Finally, if \(A \), \( B \in \mathsf{Syn}(\mathcal{F}^*, \mathcal{G}) \), then \( \overline{\mathcal{F}} \cdot \overline{\mathcal{G}} \subseteq \overline{A} \cap \overline{B} = \overline{A \cap B} \) and hence \( A \cap B \in \mathsf{Syn}(\mathcal{F}^*, \mathcal{G}) \).
    
    Also from \cref{lemma:algebraic-characterizations}(c\(^ \prime \)) we have \( \overline{\mathcal{F}} \cdot \overline{\mathcal{G}} \subseteq \overline{\mathsf{Syn}(\mathcal{F}^*, \mathcal{G})} \), and since \( \overline{\mathsf{Syn}(\mathcal{F}^*, \mathcal{G})} \) is closed we have \( c\ell(\overline{\mathcal{F}} \cdot \overline{\mathcal{G}}) \subseteq \overline{\mathsf{Syn}(\mathcal{F}^*, \mathcal{G})}\).
    To see the reverse inclusion \( \overline{\mathsf{Syn}(\mathcal{F}^*, \mathcal{G})} \subseteq c\ell(\overline{\mathcal{F}} \cdot \overline{\mathcal{G}}) \) let \( p \in \overline{\mathsf{Syn}(\mathcal{F}^*, \mathcal{G})} \) and \( A \in p \).
    By \cref{prop:mesh-operator}(g), \cref{proposition:duality-principle}, and \cref{lemma:algebraic-characterizations}(c) we have  \( (\overline{\mathcal{F}} \cdot \overline{\mathcal{G}}) \cap \overline{A} \ne \emptyset \).
    Hence \( p \in c\ell(\overline{\mathcal{F}} \cdot \overline{\mathcal{G}}) \).
    
    \textbf{(b)}
    The proof of this statement is similar to (a).
    
    First, recall again from \cite[Lemma~5.15]{Berglund1984}, we have \( \mathcal{F} \cdot \mathcal{G} \) is a filter on \( S \) and so \( \overline{\mathcal{F} \cdot \mathcal{G}} \) is nonempty closed subset of \( \beta S \).
    From \cref{lemma:algebraic-characterizations}(b) we have \( \mathsf{Thick}(\mathcal{F}, \mathcal{G}^*)  = \{ A \subseteq S : \overline{\mathcal{F} \cdot \mathcal{G}} \subseteq \overline{A} \} \) and we can use this characterization to verify \( \mathsf{Thick}(\mathcal{F}, \mathcal{G}^*) \) is a filter on \( S \).
    We have \( S \in \mathsf{Thick}(\mathcal{F}, \mathcal{G}^*) \) (since \( \overline{\mathcal{F} \cdot \mathcal{G}} \subseteq \beta S = \overline{S} \)) and \( \emptyset \not\in \mathsf{Thick}(\mathcal{F}, \mathcal{G}^*) \) (since \( \overline{\emptyset} = \emptyset \)).
    If \( A \in \mathsf{Thick}(\mathcal{F}, \mathcal{G}^*) \) and \( A \subseteq B \subseteq S \), then \( \overline{\mathcal{F} \cdot \mathcal{G}} \subseteq \overline{A} \subseteq \overline{B} \) and hence \( B \in \mathsf{Thick}(\mathcal{F}, \mathcal{G}^*) \).
    If \( A, B \in \mathsf{Thick}(\mathcal{F}, \mathcal{G}^*) \), then \( \overline{\mathcal{F} \cdot \mathcal{G}} \subseteq \overline{A} \cap \overline{B} = \overline{A \cap B} \) and hence \( A \cap B \in \mathsf{Thick}(\mathcal{F}, \mathcal{G}^*) \). 
    
    Therefore we have \( \overline{\mathcal{F} \cdot \mathcal{G}} \subseteq \overline{\mathsf{Thick}(\mathcal{F}, \mathcal{G}^*)} \).
    To see the reverse inclusion \( \overline{\mathsf{Thick}(\mathcal{F}, \mathcal{G}^*)} \subseteq \overline{\mathcal{F} \cdot \mathcal{G}} \) let \( p \in \overline{\mathsf{Thick}(\mathcal{F}, \mathcal{G})} \) and let \( A \in p \).
    Then by \cref{prop:mesh-operator}(g), \cref{proposition:duality-principle}, and \cref{lemma:algebraic-characterizations}(b\( ^\prime \)) we have \( \overline{\mathcal{F} \cdot \mathcal{G}} \cap \overline{A} \ne \emptyset \). 
    Hence \( p \in \overline{\mathcal{F} \cdot \mathcal{G}} \).
\end{proof}

In \cite[Theorem~2.3]{davenport:1987}, Davenport and Hindman obtained a special case of \cref{theorem:important-filters}(a).
In our notation, they proved that  \( \mathsf{Syn}(p, \mathcal{C}) \) is a filter, where  \( \mathcal{C} \) is the cofinite filter on  \( \mathbb{N} \).
Protasov implicitly uses  \( \mathsf{Syn}(\mathcal{C}^*, \mathcal{C}) \), where  \( \mathcal{C} \) is the cofinite filter on an infinite group  \( G \) to characterize  \( c\ell_{\beta G}( G^* \cdot G^*) \) \cite[Theorem~3.20]{Protasov2011}, where  \( G^* \) is the collection of all non-principal ultrafilters on  \( G \).
Finally, we also note that the inclusion (again, in our notation) \( \mathsf{Thick}(\mathcal{F}, \mathcal{G}^*) \subseteq \mathsf{Syn}(\mathcal{F}^*, \mathcal{G}) \) was first proved by Berglund and Hindman \cite[Lemma~5.15]{Berglund1984}.
% Hindman and \textcolor{red}{XX} proves \textcolor{red}{XX} for \textcolor{red}{a theorem I need to enter}. What does this refer to?

As a consequence of \cref{theorem:important-filters}(a) is we can easily characterize closed subsemigroups, left ideals, and right ideals of  \( \beta S \). 
These results were previously proved by Davenport in \cite{Davenport:1990wq} and, independently, by Papazyan in \cite{Papazyan1990}.
We'll derive this characterizations from a more general result:
\begin{theorem}%[\textit{Product of two compact sets}]
  \label{theorem:combinatorial-characterization-of-products}
  Let  \( \mathcal{F} \),  \( \mathcal{G} \), and  \( \mathcal{H} \) be filters on a semigroup  \( S \).
  Then  \( \overline{\mathcal{F}} \cdot \overline{\mathcal{G}} \subseteq \overline{\mathcal{H}} \) if and only if  \( \mathcal{H} \subseteq \mathsf{Syn}(\mathcal{F}^*, \mathcal{G}) \).
\end{theorem}
\begin{proof}
  Since  \( \overline{\mathcal{H}} \) is closed, we have  \( \overline{\mathcal{F}} \cdot \overline{\mathcal{G}} \subseteq \overline{\mathcal{H}} \iff c\ell(\overline{\mathcal{F}} \cdot \overline{\mathcal{G}}) \subseteq \overline{\mathcal{H}} \iff \overline{\mathsf{Syn}(\mathcal{F}^*, \mathcal{G})} \subseteq \overline{\mathcal{H}} \iff \mathcal{H} \subseteq \mathsf{Syn}(\mathcal{F}^*, \mathcal{G}) \), where the middle and last equivalences use \cref{theorem:important-filters}(a).
\end{proof}

\begin{corollary}%[\textit{Characterizations of compact subsemigroups, left ideals, right ideals, and two-sided ideals}]
  \label{corollary:combinatorial-characterizations}
  Let  \( \mathcal{F} \) be a filter on a semigroup  \( S \).
  \begin{itemize}
    \item[(a)]  \( \overline{\mathcal{F}} \) is a closed subsemigroup of  \( \beta S \) if and only if  \( \mathcal{F} \subseteq \mathsf{Syn}(\mathcal{F}^*, \mathcal{F}) \). 
    \item[(b)]  \( \overline{\mathcal{F}} \) is a closed left ideal of  \( \beta S \) if and only if  \( \mathcal{F} \subseteq \mathsf{Syn}(\{S\}^*, \mathcal{F}) \). 
    \item[(c)]  \( \overline{\mathcal{F}} \) is a closed right ideal of  \( \beta S \) if and only if  \( \mathcal{F} \subseteq \mathsf{Syn}(\mathcal{F}^*, \{S\}) \). 
    \item[(d)]  \( \overline{\mathcal{F}} \) is a closed (two-sided) ideal of  \( \beta S \) if and only if  \( \mathcal{F} \subseteq \mathsf{Syn}(\{S\}^*, \mathcal{F}) \cap \mathsf{Syn}(\mathcal{F}^*, \{S\}) \). 
  \end{itemize}
\end{corollary}
\begin{proof}
    In the justifications of (a), (b), and (c) below each of the second equivalences follows since \( \overline{\mathcal{F}} \) is closed and each of the third equivalences follows from  \cref{theorem:combinatorial-characterization-of-products}:
    \begin{itemize}
        \item[\textbf{(a)}] \(\overline{\mathcal{F}} \) is a subsemigroup \( \iff \overline{\mathcal{F}} \cdot \overline{\mathcal{F}} \subseteq \overline{\mathcal{F}} \iff  c\ell(\overline{\mathcal{F}} \cdot \overline{\mathcal{F}}) \subseteq \overline{\mathcal{F}} \iff \mathcal{F} \subseteq \mathsf{Syn}(\mathcal{F}^*, \mathcal{F}) \).
        
        \item[\textbf{(b)}] \(\overline{\mathcal{F}} \) is a left ideal \( \iff \beta S \cdot \overline{\mathcal{F}} \subseteq \overline{\mathcal{F}} \iff  c\ell(\beta S \cdot \overline{\mathcal{F}}) \subseteq \overline{\mathcal{F}} \iff \mathcal{F} \subseteq \mathsf{Syn}(\{S\}^*, \mathcal{F}) \).
         
        \item[\textbf{(c)}] \(\overline{\mathcal{F}} \) is a right ideal \( \iff \overline{\mathcal{F}} \cdot \beta S \subseteq \overline{\mathcal{F}} \iff  c\ell(\overline{\mathcal{F}} \cdot \beta S) \subseteq \overline{\mathcal{F}} \iff \mathcal{F} \subseteq \mathsf{Syn}(\mathcal{F}^*, \{S\}) \).
         
         \item[\textbf{(d)}]
         This follows directly from statements~(b) and (c).
    \end{itemize}
\end{proof}

\section{Relative notions of piecewise syndetic sets}
\label{section:relative-piecewise-syndetic}
In this section, inspired by \cref{corollary:brown-lemma}, we define relative piecewise syndetic sets in such a way that \emph{any} such collection forms a grill on \( S \). 
% In this section, inspired by the result of Bergelson, Hindman, and McCutcheon \cite[Theorem~2.4(d)]{Bergelson1998} and its use in \cref{corollary:brown-lemma} to show \( \mathsf{PS} \) is a grill, we define relative piecewise syndetic sets in such a way that \emph{any} such collection forms a grill on \( S \). 
We show, under the conditions considered in \cite{Shuungula:2009ty}, that our definition of relative piecewise syndetic satisfies Shuungula, Zelenyuk, and Zelenyuk's earlier previously defined notion of relative piecewise syndetic.
Whether the converse implication is true is stated as an open question.

\begin{definition}%[\textit{Relative notion of piecewise syndetic}]
  \label{definition:piecewise-FG-syndetic}
  Let \( A \subseteq S \) and let \( \mathcal{F} \) and \( \mathcal{G} \) both be stacks on \( S \).
  \begin{itemize}
    \item[(a)] \( A \) is \define{piecewise \( (\mathcal{F}, \mathcal{G}) \)-syndetic} if and only if there exist \( B \in \mathsf{Syn}(\mathcal{F}, \mathcal{G}) \) and \( C \in \mathsf{Thick}(\mathcal{F}, \mathcal{G}) \) such that \( A = B \cap C \).
    \item[(b)] We also define the collection \( \mathsf{PS}(\mathcal{F}, \mathcal{G}) = \{ A \subseteq S : A \text{ is piecewise } (\mathcal{F}, \mathcal{G})\text{-syndetic} \} \).
  \end{itemize}
\end{definition}

Note that we have \( \mathsf{PS} = \mathsf{PS}(\{S\}, \{S\}) \).
Of course, our main motivation in defining piecewise \((\mathcal{F}, \mathcal{G})\)-syndetic sets as an intersection of \( (\mathcal{F}, \mathcal{G}) \)-syndetic and \( (\mathcal{F}, \mathcal{G}) \)-thick sets is we can apply \cref{prop:mesh-operator}(h) to obtain a generalization of \cref{corollary:brown-lemma}:
\begin{theorem}%[\textit{Relative piecewise syndetic sets are partition regular}]
  \label{theorem:brown-lemma-relative-piecewise-syndetic}
  Let \( \mathcal{F} \) and \( \mathcal{G} \) both be stacks on \( S \).
  Then \( \mathsf{PS}(\mathcal{F}, \mathcal{G}) \) is a grill on \( S \) with \( \mathsf{Syn}(\mathcal{F}, \mathcal{G}) \subseteq \mathsf{PS}(\mathcal{F}, \mathcal{G}) \) and \( \mathsf{Thick}(\mathcal{F}, \mathcal{G}) \subseteq \mathsf{PS}(\mathcal{F}, \mathcal{G}) \).
\end{theorem}
\begin{proof}
  By definition \( \mathsf{PS}(\mathcal{F}, \mathcal{G}) = \{ B \cap C : B \in \mathsf{Syn}(\mathcal{F}, \mathcal{G}) \mbox{ and } C \in \mathsf{Thick}(\mathcal{F}, \mathcal{G}) \} \), and
  by \cref{proposition:duality-principle} we have \( \mathsf{Syn}(\mathcal{F}, \mathcal{G}) = \mathsf{Thick}(\mathcal{F}, \mathcal{G})^* \).
  Hence from \cref{prop:mesh-operator}(h) it follows that \( \mathsf{PS}(\mathcal{F}, \mathcal{G}) \) is a grill on \( S \) that contains both \( \mathsf{Syn}(\mathcal{F}, \mathcal{G}) \) and \( \mathsf{Thick}(\mathcal{F}, \mathcal{G}) \).
\end{proof}

Similar to \cref{lemma:algebraic-characterizations}, if we assume \( \mathcal{F} \) and \( \mathcal{G} \) are both filters on \( S \), then we can (partially) solve a few special instances of the classification problem for relative piecewise syndetic sets:

\begin{proposition}
    \label{proposition:algebraic-characterization-relative-PS-easy-cases}
    Let \( \mathcal{F} \) and \( \mathcal{G} \) both be filters on \( S \).
    \begin{itemize}
        \item[(a)] \( \mathsf{PS}(\mathcal{F}, \mathcal{G}) \subseteq \mathsf{Thick}(\mathcal{F}^*, \mathcal{G}) \) and the inclusion can be strict.
        
        \item[(b)] \( \mathsf{PS}(\mathcal{F}, \mathcal{G}^*) = \mathsf{Syn}(\mathcal{F}, \mathcal{G}^*) \)
    
        \item[(c)] \( \mathsf{PS}(\mathcal{F}^*, \mathcal{G}) = \mathsf{Thick}(\mathcal{F}^*, \mathcal{G}) \)
    
        \item[(d)] \( \mathsf{PS}(\mathcal{F}^*, \mathcal{G}^*) \subseteq \{ A \subseteq S : \mbox{ there exists } p \in \overline{\mathcal{F}} \mbox{ such that } \overline{p \cdot \mathcal{G}}  \cap \overline{A} \ne \emptyset \} \).% and the inclusion can be strict.
  \end{itemize}    
\end{proposition}
\begin{proof}
    \textbf{(a)} 
     Observe, via \cref{lemma:algebraic-characterizations}(c), that
    \(
        \{ A \subseteq S : (\exists q \in \overline{\mathcal{G}})\; \overline{\mathcal{F}} \cdot q \cap \overline{A} \ne \emptyset \} = \mathsf{Thick}(\mathcal{F}^*, \mathcal{G})
    \).
    
    Now let \( A \in \mathsf{PS}(\mathcal{F}, \mathcal{G}) \) and pick \( B \in \mathsf{Syn}(\mathcal{F}, \mathcal{G}) \) and \( C \in \mathsf{Thick}(\mathcal{F}, \mathcal{G}) \) as guaranteed for \( A \).
    By \cref{lemma:algebraic-characterizations}(a), pick \( q \in \overline{\mathcal{G}} \) with \( \overline{\mathcal{F}} \cdot q \subseteq \overline{C} \).
    Then
    \[
        (\overline{\mathcal{F}} \cdot q) \cap \overline{A} =  (\overline{\mathcal{F}} \cdot q) \cap \overline{B \cap C} = (\overline{\mathcal{F}} \cdot q) \cap \overline{B} \cap \overline{C} = (\overline{\mathcal{F}} \cdot q) \cap \overline{B} \ne \emptyset,
    \]
    where the last relation follows from \cref{lemma:algebraic-characterizations}(a\( ^\prime \)).
    
    To see that the inclusion can be strict, in \( (\mathbb{N}, +) \) we note
    \(
        \mathsf{PS} \subsetneq \mathsf{Thick}(\{\mathbb{N}\}^*, \{\mathbb{N}\})
    \) since, for example, \( \{2\} \in \mathsf{Thick}(\{\mathbb{N}\}^*, \{\mathbb{N}\}) \) but \( \{2\} \not\in \mathsf{PS} \).
    
  \textbf{(b)} 
  By \cref{theorem:brown-lemma-relative-piecewise-syndetic} it suffices to verify \( \mathsf{PS}(\mathcal{F}, \mathcal{G}^*) \subseteq \mathsf{Syn}(\mathcal{F}, \mathcal{G}^*) \).

  Let \( A \in \mathsf{PS}(\mathcal{F}, \mathcal{G}^*) \) and pick \( B \in \mathsf{Syn}(\mathcal{F}, \mathcal{G}^*) \) and \( C \in \mathsf{Thick}(\mathcal{F}, \mathcal{G}^*) \) as guaranteed for \( A \).
  Then
  \[
    \overline{\mathcal{F} \cdot \mathcal{G}} \cap \overline{A} = \overline{\mathcal{F} \cdot \mathcal{G}} \cap \overline{B \cap C} = \overline{\mathcal{F} \cdot \mathcal{G}} \cap \overline{B} \cap \overline{C} = \overline{\mathcal{F} \cdot \mathcal{G}} \cap \overline{B} \ne \emptyset,
  \]
  where the third equality follows from \cref{lemma:algebraic-characterizations}(b) and last relation follows from \cref{lemma:algebraic-characterizations}(b\( ^\prime \)).
  Hence by \cref{lemma:algebraic-characterizations}(b\( ^\prime \)) we have \( A \in \mathsf{Syn}(\mathcal{F}, \mathcal{G}^*) \).

    \textbf{(c)}
    By \cref{theorem:brown-lemma-relative-piecewise-syndetic} it suffices to verify \( \mathsf{PS}(\mathcal{F}^*, \mathcal{G}) \subseteq \mathsf{Thick}(\mathcal{F}^*, \mathcal{G}) \).
    
    Let \( A \in \mathsf{PS}(\mathcal{F}^*, \mathcal{G}) \) and pick \( B \in \mathsf{Syn}(\mathcal{F}^*, \mathcal{G}) \) and \( C \in \mathsf{Thick}(\mathcal{F}^*, \mathcal{G}) \) as guaranteed for \( A \).
    Then
    \[ \overline{\mathcal{F}} \cdot \overline{\mathcal{G}} \cap \overline{A} = \overline{\mathcal{F}} \cdot \overline{\mathcal{G}} \cap \overline{B \cap C} = \overline{\mathcal{F}} \cdot \overline{\mathcal{G}} \cap \overline{B} \cap \overline{C} = \overline{\mathcal{F}} \cdot \overline{\mathcal{G}} \cap \overline{C} \ne \emptyset,
    \]
    where the third equality follows from \cref{lemma:algebraic-characterizations}(c\( ^\prime \)) and last relation follows from \cref{lemma:algebraic-characterizations}(c).
  Hence by \cref{lemma:algebraic-characterizations}(c) we have \( A \in \mathsf{Thick}(\mathcal{F}^*, \mathcal{G}) \).

    \textbf{(d)}
    Let \( A \in \mathsf{PS}(\mathcal{F}^*, \mathcal{G}^*) \) and pick \( B \in \mathsf{Syn}(\mathcal{F}^*, \mathcal{G}^*) \) and \( C \in \mathsf{Thick}(\mathcal{F}^*, \mathcal{G}^*) \) as guaranteed for \( A \).
    By \cref{lemma:algebraic-characterizations}(d) pick \( p \in \overline{\mathcal{F}} \) such that \( \overline{p \cdot \mathcal{G}} \subseteq \overline{C} \).
    Then
    \[
        \overline{p \cdot \mathcal{G}} \cap \overline{A} = \overline{p \cdot \mathcal{G}} \cap \overline{B \cap C} = \overline{p \cdot \mathcal{G}} \cap \overline{B} \cap \overline{C} = \overline{p \cdot \mathcal{G}} \cap \overline{B} \ne \emptyset,
    \]
    where the last relation follows \cref{lemma:algebraic-characterizations}(d\( ^\prime \)).
\end{proof}

The point of statements (b) and (c) is that for filters \( \mathcal{F} \) and \( \mathcal{G} \), neither \( \mathsf{PS}(\mathcal{F}, \mathcal{G}^*) \) nor \( \mathsf{PS}(\mathcal{F}^*, \mathcal{G}) \) produce any new notion of size beyond relative syndetic and thick sets, respectively. 
We suspect that the inclusion in statement (d) is strict even in \( (\mathbb{N}, +) \), but we don't know of an example to prove it.
Hence from our point-of-view, statement (a) represents the main interesting new notion of size and for the rest of this section we'll restrict our attention to \( \mathsf{PS}(\mathcal{F}, \mathcal{G}) \) for filters \( \mathcal{F} \) and \( \mathcal{G} \).

Observe from \cref{lemma:algebraic-characterizations}(a) we have \(\mathsf{Thick} = \bigcup_{q \in \beta S} \mathsf{Thick}(\{S\}, q) \), and so from \cref{proposition:characterization-of-PS-as-broken-syndetic}(a) we can conclude \( A \in \mathsf{PS} \) if there exists \( q \in \beta S \) with \( A \in \mathsf{Syn}(\{S\}, \mathsf{Thick}(\{S\}, q)) \).

% Recall that Bergelson, Hindman, and McCutcheon \cite[Theorem~2.4(d)]{Bergelson1998} in our notation states \( \mathsf{PS} = \mathsf{Syn}(\{S\}, \mathsf{Thick}) \)

% proved \( A \in \mathsf{PS} \) if and only if there exists \( q \in \beta S \) with \( A \in \mathsf{Syn}(\{S\}, \mathsf{Thick}(\{S\}, q)) \).

For a filter \( \mathcal{F} \) such that \( \overline{\mathcal{F}} \) is a closed subsemigroup of \( \beta S \), the notion of a piecewise \( \mathcal{F} \)-syndetic set was defined earlier by Shuungula, Zelenyuk, and Zelenyuk \cite[p.~534, second paragraph]{Shuungula:2009ty} as \( A \subseteq S \) is \define{piecewise \( \mathcal{F} \)-syndetic} if and only if there exists \( q \in \overline{\mathcal{F}} \) such that \( A \in \mathsf{Syn}(\mathcal{F}, \mathsf{Thick}(\mathcal{F}, q)) \).
It's not immediately clear, in this case, that the two definitions of relative piecewise syndetic sets are equivalent.
We prove, under the conditions considered in their paper, that \( A \in \mathsf{PS}(\mathcal{F}, \mathcal{F}) \) implies \( A \) is piecewise \( \mathcal{F} \)-syndetic (the converse implication is open).
We'll derive this from a more general result:
\begin{theorem}
	Let \( \mathcal{F} \) and \( \mathcal{G} \) both be filters on \( S \) with \( \overline{\mathcal{F}} \) is a closed subsemigroup of \( \beta S \) and \( \overline{\mathcal{F}} \cdot \overline{\mathcal{G}} \subseteq \overline{\mathcal{G}} \).
	If \( A \in \mathsf{PS}(\mathcal{F}, \mathcal{G}) \), then there exists \( q \in \overline{\mathcal{G}} \) with \( A \in \mathsf{Syn}(\mathcal{F}, \mathsf{Thick}(\mathcal{F}, q)) \).
\end{theorem} 
\begin{proof}
	Pick \( B \in \mathsf{Syn}(\mathcal{F}, \mathcal{G}) \) and \( C \in \mathsf{Thick}(\mathcal{F}, \mathcal{G}) \) such that \( A = B \cap C \).
	By \cref{lemma:algebraic-characterizations}(a) pick \( q \in \overline{\mathcal{G}} \) such that \( \overline{\mathcal{F}} \cdot q \subseteq \overline{C} \).
	Since \( \overline{\mathcal{F}} \) is a subsemigroup we have \( \overline{\mathcal{F}} \cdot \overline{\mathcal{F}} \cdot q \subseteq \overline{\mathcal{F}} \cdot q \subseteq \overline{C} \).
	Since \( \overline{\mathcal{F}} \cdot \overline{\mathcal{G}} \subseteq \overline{\mathcal{G}} \) and \( B \in \mathsf{Syn}(\mathcal{F}, \mathcal{G}) \), by \cref{lemma:algebraic-characterizations}(a\( ^\prime \)) it follows that for all \( p \in \overline{\mathcal{F}} \) we have \( \overline{\mathcal{F}} \cdot p \cdot q \cap \overline{B} \ne \emptyset \).
	% 	Therefore by \cref{lemma:algebraic-characterizations}(a\( ^\prime \)) it follows \( B \in \mathsf{Syn}(\mathcal{F}, \mathsf{Thick}(\mathcal{F}, q)) \).
	Hence for all \( p \in \overline{\mathcal{F}} \) we have
	\begin{align*}
		(\overline{\mathcal{F}} \cdot p \cdot q)  \cap \overline{A} &= (\overline{\mathcal{F}} \cdot p \cdot q) \cap \overline{B \cap C} \\
		&=  (\overline{\mathcal{F}} \cdot p \cdot q)  \cap \overline{B} \cap \overline{C} \\
		&= (\overline{\mathcal{F}} \cdot p \cdot q) \cap \overline{B}  \ne \emptyset.
	\end{align*}
\end{proof}
\begin{corollary}
	\label{corollary:relative-PS-union-are-thick}
	Let \( \mathcal{F} \) be a filter on \( S \) with \( \overline{\mathcal{F}} \) a closed subsemigroup of \( \beta S \).
	If \( A \in \mathsf{PS}(\mathcal{F}, \mathcal{F}) \), then there exists \( q \in \overline{\mathcal{F}} \) with \( A \in \mathsf{Syn}(\mathcal{F}, \mathsf{Thick}(\mathcal{F}, q)) \).
\end{corollary}

\begin{question}
	\label{question:converse-PS}
	Let \( \mathcal{F} \) be a filter on \( S \) with  \( \overline{\mathcal{F}} \)	a closed subsemigroup of \( \beta S \).
	If \( A \subseteq S \) such that there exists \( q \in \overline{\mathcal{F}} \) with \( A \in \mathsf{Syn}(\mathcal{F}, \mathsf{Thick}(\mathcal{F}, q)) \), must \( A \) be a member of \( \mathsf{PS}(\mathcal{F}, \mathcal{F}) \)?	
% 	Is the following converse of \cref{corollary:relative-PS-union-are-thick} true?
% 	\begin{quote}
% 	Let \( \mathcal{F} \) be a filter on \( S \) with \( \overline{\mathcal{F}} \)	a closed subsemigroup of \( \beta S \).
% 	If \( A \subseteq S \) such that there exists \( q \in \overline{\mathcal{F}} \) with \( A \in \mathsf{Syn}(\mathcal{F}, \mathsf{Thick}(\mathcal{F}, q)) \), then \( A \in \mathsf{PS}(\mathcal{F}, \mathcal{F}) \).
% 	\end{quote}
\end{question}

\subsection*{Acknowledgements}
\label{section:acknowledgements}

We gratefully acknowledge and thank Florian Richter for several important discussions along with valuable feedback and suggestion on earlier drafts of this article.
We thank the referee for a careful reading and several suggestions that improved the exposition.
We thank Jessica Christian for feedback on drafts of the introduction.
We thank Baglini, Blass, and Di Nasso for helpful conversations on their work related to finite embeddability, and we also thank Anush Tseurunyan and Andrew Zucker for helpful discussions and interest.
Finally, we also thank Vitaly Bergelson and Neil Hindman for helpful correspondence.

\printbibliography
\end{document}